\documentclass[a4paper,11pt]{amsart}
\usepackage{amsmath,amssymb,amsfonts,latexsym}

\newtheorem{theorem}{Theorem}[section]
\newtheorem{lemma}[theorem]{Lemma}
\newtheorem{corollary}[theorem]{Corollary}

\input{tcilatex}

\begin{document}
\title[$q$-\textbf{Genocchi numbers and polynomials of higher order}]{%
\textbf{A new family of }$q$\textbf{-analogue of Genocchi numbers and
polynomials of higher order}}
\author{\textbf{Serkan Araci}}
\address{\textbf{University of Gaziantep, Faculty of Science and Arts,
Department of Mathematics, 27310 Gaziantep, TURKEY}}
\email{\textbf{mtsrkn@hotmail.com}}
\author{\textbf{Mehmet Acikgoz}}
\address{\textbf{University of Gaziantep, Faculty of Science and Arts,
Department of Mathematics, 27310 Gaziantep, TURKEY}}
\email{\textbf{acikgoz@gantep.edu.tr}}
\author{\textbf{Jong Jin Seo}}
\address{\textbf{Department of Applied Mathematics, Pukyong National
University, Busan 608-737, Republic of KOREA}}
\email{\textbf{seo2011@pknu.ac.kr}}

\begin{abstract}
The new $q$-Euler polynomials was introduced by T. Kim in \textquotedblleft $%
q$-Generalized Euler numbers and polynomials, Russian Journal of
Mathematical Physics, Vol. 13, No. 3, 2006, pp. 293-308\textquotedblright\
by means of the following generating function:%
\begin{equation*}
\sum_{j=0}^{\infty }\frac{z^{j}}{\left[ j\right] _{q}!}E_{j,q}\left(
x\right) =\frac{\left[ 2\right] _{q}}{e_{q}\left( z\right) +1}e_{q}\left(
xz\right) \text{.}
\end{equation*}

In this work, we consider the generating function of Kim's $q$-Euler
polynomials and introduce new generalization of $q$-Genocchi polynomials and
numbers of higher order. Also, we give surprising identities for studying in
Analytic Numbers Theory and especially in Mathematical Physics. Moreover, by
applying $q$-Mellin transformation to generating function of $q$-Genocchi
polynomials of higher order and so we define $q$-Hurwitz-Zeta type function
which interpolates of this polynomials at negative integers.

\vspace{2mm}\noindent \textsc{2010 Mathematics Subject Classification.}
11S80, 11B68.

\vspace{2mm}

\end{abstract}

\thanks{.$\sim $\textbf{25th International Conference of the Jangjeon
Mathematical Society}$\sim $}
\maketitle




\section{\textbf{Introduction}}


Throughout this work, we assume that $q\in 
\mathbb{C}
$ with $\left\vert q\right\vert <1$. The $q$-integer of $x$ is defined by $%
\left[ x\right] _{q}=\frac{1-q^{x}}{1-q}$ and note that $\lim_{q\rightarrow
1}\left[ x\right] _{q}=x$. The $q$-derivative is defined by F. H. Jackson as
follows:%
\begin{equation}
D_{q}f\left( x\right) =\frac{d}{d_{q}x}f\left( x\right) =\frac{f\left(
x\right) -f\left( qx\right) }{\left( 1-q\right) x}\text{.}
\label{Equation 2}
\end{equation}

Taking $f\left( x\right) =x^{n}$ in (\ref{Equation 2}), it becomes as
follows:%
\begin{equation*}
D_{q}x^{n}=\frac{x^{n}-\left( qx\right) ^{n}}{\left( 1-q\right) x}=\left[ n%
\right] _{q}x^{n-1}\text{ and }\left( \frac{d}{d_{q}x}\right) ^{n}f\left(
x\right) =\left[ n\right] _{q}!
\end{equation*}

where $\left[ n\right] _{q}!=\left[ n\right] _{q}\left[ n-1\right]
_{q}\cdots \left[ 1\right] _{q}$. Now, we give definitions of two kinds of $%
q $-exponential functions as follows:

For any $z\in 
\mathbb{C}
$ with $\left\vert z\right\vert <1$, 
\begin{equation}
e_{q}\left( z\right) =\sum_{l=0}^{\infty }\frac{z^{l}}{\left[ l\right] _{q}!}%
\text{ and }E_{q}\left( z\right) =\sum_{l=0}^{\infty }q^{\binom{l}{2}}\frac{%
z^{l}}{\left[ l\right] _{q}!}.  \label{Equation 1}
\end{equation}

By (\ref{Equation 1}), it is not difficult to show that $\left[ l\right] _{%
\frac{1}{q}}!=q^{-\binom{l}{2}}\left[ l\right] _{q}!$. Then, we have the
following%
\begin{equation}
e_{\frac{1}{q}}\left( z\right) =E_{q}\left( z\right) \text{.}
\label{Equation 13}
\end{equation}

For the $q$-commuting variables $x$ and $y$ such that $yx=qxy$, we know that 
\begin{equation}
e_{q}\left( x+y\right) =e_{q}\left( x\right) e_{q}\left( y\right) \text{.}
\label{Equation 20}
\end{equation}

The $q$-integral was defined by Jackson as follows:%
\begin{equation}
\int_{0}^{x}f\left( \xi \right) d_{q}\xi =\left( 1-q\right)
x\sum_{l=0}^{\infty }f\left( q^{l}x\right) q^{l}  \label{Equation 21}
\end{equation}

provided that the series on the right hand side converges absolutely.

In particular, if $f(\xi )=\xi ^{n}$, then we have%
\begin{equation}
\int_{0}^{x}\xi ^{n}d_{q}\xi =\frac{1}{\left[ n+1\right] _{q}}x^{n+1}\text{.}
\label{Equation 22}
\end{equation}

The definitions of $q$-integral and $q$-derivative imply the following
formula:%
\begin{equation}
D_{q}\left( \int_{0}^{x}f\left( t\right) d_{q}t\right) =f\left( x\right)
\label{Equation 23}
\end{equation}

and%
\begin{equation}
D_{q}\left( f\left( x\right) g\left( x\right) \right) =f\left( x\right)
D_{q}\left( g\left( x\right) \right) +g\left( qx\right) D_{q}\left( f\left(
x\right) \right) \text{.}  \label{Equation 24}
\end{equation}

For more informations of Eqs. (1-8), you can refer to [19-23].

The ordinary Euler numbers and polynomials are defined via the following
generating function:%
\begin{equation*}
e^{E\left( x\right) t}=\sum_{n=0}^{\infty }E_{n}\left( x\right) \frac{t^{n}}{%
n!}=\frac{2}{e^{t}+1}e^{xt}\text{, }\left\vert t\right\vert <\pi \text{ }
\end{equation*}

where the usual convention about replacing $E^{n}\left( x\right) $ by $%
E_{n}\left( x\right) $ (see \cite{Bayad}, \cite{Bayad 1}, \cite{Ryoo}, \cite%
{Jolany}, \cite{Kim 1}, \cite{Ozden}).

In \cite{Kim 1}, the new $q$-generalization of Euler polynomials are
introduced by T. Kim as follows:%
\begin{equation*}
\sum_{j=0}^{\infty }\frac{z^{j}}{\left[ j\right] _{q}!}E_{j,q}\left(
x\right) =\frac{\left[ 2\right] _{q}}{e_{q}\left( z\right) +1}e_{q}\left(
xz\right) \text{.}
\end{equation*}

\bigskip By using the above generating function, Kim gave some interesting
and fascinating properties for new $q$-generalization of Euler numbers and
polynomials. We note that these polynomials are used to study in Analytic
Numbers Theory. So, in the next section, we shall introduce generating
function of $q$-Genocchi numbers and polynomials of higher order.
Additionally, we shall give their applications.

\section{\textbf{New }$q$\textbf{-Genocchi numbers and polynomials of higher
order}}

In this section, we introduce generating function for $q$-Genocchi
polynomials\ of higher order by using Kim's method in \cite{Kim 1}. Thus, we
now start as follows: 
\begin{equation}
S_{q}\left( t:\alpha \right) =\sum_{n=0}^{\infty }\frac{t^{n}}{\left[ n%
\right] _{q}!}G_{n,q}^{\left( \alpha \right) }\text{.}  \label{Equation 3}
\end{equation}

Here $G_{n,q}^{\left( \alpha \right) }$ is called as the $q$-Genocchi
numbers of higher order. By using $q$-derivative operator, we compute as
follows:%
\begin{equation}
S_{q}\left( \frac{d}{d_{q}x}:\alpha \right) x^{k}=\sum_{n=0}^{\infty }\frac{1%
}{\left[ n\right] _{q}!}G_{n,q}^{\left( \alpha \right) }\left( \frac{d}{%
d_{q}x}\right) ^{n}x^{k}=\sum_{n=0}^{k}\binom{k}{n}_{q}G_{n,q}^{\left(
\alpha \right) }x^{k-n}\text{,}  \label{Equation 14}
\end{equation}

where 
\begin{equation*}
\binom{k}{n}_{q}=\frac{\left[ k\right] _{q}\left[ k-1\right] _{q}\cdots %
\left[ k-n+1\right] _{q}}{\left[ n\right] _{q}!}\text{.}
\end{equation*}

Similarly, by (\ref{Equation 14}), we develop as follows:%
\begin{eqnarray*}
S_{q}\left( \frac{d}{d_{q}x}:\alpha \right) e_{q}\left( tx\right)
&=&\sum_{j=0}^{\infty }\frac{G_{j,q}^{\left( \alpha \right) }}{\left[ j%
\right] _{q}!}\left( \frac{d}{d_{q}x}\right) ^{j}\sum_{k=0}^{\infty }\frac{%
x^{k}}{\left[ k\right] _{q}!}t^{k} \\
&=&\sum_{j=0}^{\infty }\frac{t^{j}}{\left[ j\right] _{q}!}G_{j,q}^{\left(
\alpha \right) }\left( x\right) \\
&=&S_{q}\left( x,t:\alpha \right) \text{.}
\end{eqnarray*}

From this point of view, we can also consider the $q$-Genocchi polynomials
of higher order in the form:%
\begin{equation}
\sum_{n=0}^{\infty }\frac{z^{n}}{\left[ n\right] _{q}!}G_{n,q}^{\left(
\alpha \right) }\left( x\right) =\left( \frac{\left[ 2\right] _{q}z}{%
e_{q}\left( z\right) +1}\right) ^{\alpha }e_{q}\left( zx\right) \text{.}
\label{Equation 4}
\end{equation}

As $q\rightarrow 1$ and $\alpha =1$\ in Eq. (\ref{Equation 4}), we easily
reach the following%
\begin{equation*}
\lim_{q\rightarrow 1}G_{n,q}^{\left( 1\right) }\left( x\right) =G_{n}\left(
x\right)
\end{equation*}

which $G_{n}\left( x\right) $ is known as ordinary Genocchi polynomials (for
details, see \cite{Jolany}, \cite{Kim 3}, \cite{Rim}, \cite{Cangul}, \cite%
{Araci 1}, \cite{Araci 3}).

By (\ref{Equation 3}) and (\ref{Equation 4}), we readily see that 
\begin{equation*}
\sum_{j=0}^{\infty }\frac{z^{j}}{\left[ j\right] _{q}!}G_{j,q}^{\left(
\alpha \right) }\left( x\right) =\sum_{j=0}^{\infty }\left( \sum_{n=0}^{j}%
\binom{j}{n}_{q}x^{j-n}G_{n,q}^{\left( \alpha \right) }\right) \frac{z^{j}}{%
\left[ j\right] _{q}!}\text{.}
\end{equation*}

By comparing the coefficients of $\frac{z^{j}}{\left[ j\right] _{q}!}$ on
both sides of the above equation, then we obtain the following theorem.

\begin{theorem}
For any $j\in 
\mathbb{N}
$, we have%
\begin{equation*}
G_{j,q}^{\left( \alpha \right) }\left( x\right) =\sum_{n=0}^{j}\binom{j}{n}%
_{q}x^{j-n}G_{n,q}^{\left( \alpha \right) }\text{.}
\end{equation*}
\end{theorem}

By applying $q$-derivative operator to (\ref{Equation 4}), then we see that%
\begin{equation*}
\sum_{n=1}^{\infty }\frac{z^{n}}{\left[ n\right] _{q}!}\left\{ \frac{d}{%
d_{q}x}G_{n,q}^{\left( \alpha \right) }\left( x\right) \right\}
=z\sum_{n=0}^{\infty }\frac{z^{n}}{\left[ n\right] _{q}!}G_{n,q}^{\left(
\alpha \right) }\left( x\right) \text{.}
\end{equation*}

By comparing the coefficients of $z^{n}$ on both sides of the above
equation, we arrive the following theorem.

\begin{theorem}
For any $n\in 
\mathbb{N}
^{\ast }=\left\{ 0,1,2,3,\ldots \right\} $, we have%
\begin{equation*}
\frac{d}{d_{q}x}G_{n,q}^{\left( \alpha \right) }\left( x\right) =\left[ n%
\right] _{q}G_{n-1,q}^{\left( \alpha \right) }\left( x\right) \text{.}
\end{equation*}
\end{theorem}

For $q$-commuting variables $x$ and $y$ ($yx=qxy$), we note that%
\begin{eqnarray*}
\sum_{l=0}^{\infty }\frac{z^{l}}{\left[ l\right] _{q}!}G_{l,q}^{\left(
\alpha \right) }\left( x+y\right) &=&\left( \frac{\left[ 2\right] _{q}z}{%
e_{q}\left( z\right) +1}\right) ^{\alpha }e_{q}\left( z\left( x+y\right)
\right) \\
&=&e_{q}\left( zy\right) \left( e_{q}\left( zx\right) \left( \frac{\left[ 2%
\right] _{q}z}{e_{q}\left( z\right) +1}\right) ^{\alpha }\right) \\
&=&\sum_{l=0}^{\infty }\left( \sum_{j=0}^{l}\binom{l}{j}_{q}y^{l-j}G_{j,q}^{%
\left( \alpha \right) }\left( x\right) \right) \frac{z^{l}}{\left[ l\right]
_{q}!}\text{.}
\end{eqnarray*}

As a result, we procure the following theorem.

\begin{theorem}
For any $n\in 
\mathbb{N}
^{\ast }$, we have%
\begin{equation*}
G_{n,q}^{\left( \alpha \right) }\left( x+y\right) =\sum_{j=0}^{n}\binom{n}{j}%
_{q}y^{n-j}G_{j,q}^{\left( \alpha \right) }\left( x\right) \text{.}
\end{equation*}
\end{theorem}

By expression (\ref{Equation 4}), we compute as follows:%
\begin{eqnarray}
&&\sum_{l=0}^{\infty }\frac{z^{l}}{\left[ l\right] _{q}!}G_{l,q}^{\left(
\alpha +\beta \right) }\left( x\right)  \label{Equation 15} \\
&=&\left[ \left( \frac{\left[ 2\right] _{q}z}{e_{q}\left( z\right) +1}%
\right) ^{\alpha }\right] \left[ \left( \frac{\left[ 2\right] _{q}z}{%
e_{q}\left( z\right) +1}\right) ^{\beta }e_{q}\left( zx\right) \right] 
\notag \\
&=&\left[ \sum_{j=0}^{\infty }\frac{z^{j}}{\left[ j\right] _{q}!}%
G_{j,q}^{\left( \alpha \right) }\right] \left[ \sum_{k=0}^{\infty }\frac{%
z^{k}}{\left[ k\right] _{q}!}G_{k,q}^{\left( \beta \right) }\left( x\right) %
\right]  \notag
\end{eqnarray}

by using Cauchy product on the above equation, we derive that%
\begin{equation}
\sum_{l=0}^{\infty }\frac{z^{l}}{\left[ l\right] _{q}!}\left( \sum_{n=0}^{l}%
\binom{l}{n}_{q}G_{n,q}^{\left( \alpha \right) }G_{l-n,q}^{\left( \beta
\right) }\left( x\right) \right) \text{.}  \label{Equation 16}
\end{equation}

Comparing the coefficients of Eqs. (\ref{Equation 15}) and (\ref{Equation 16}%
), then we present the following theorem.

\begin{theorem}
For $l\in 
\mathbb{N}
^{\ast }$, then we have%
\begin{equation*}
G_{l,q}^{\left( \alpha +\beta \right) }\left( x\right) =\sum_{n=0}^{l}\binom{%
l}{n}_{q}G_{n,q}^{\left( \alpha \right) }G_{l-n,q}^{\left( \beta \right)
}\left( x\right) \text{.}
\end{equation*}
\end{theorem}

\bigskip

Jackson are defined the $q$-analogue of the Gamma function by 
\begin{equation}
\Gamma _{q}\left( x\right) =\frac{\left( q;q\right) _{\infty }}{\left(
q^{x};q\right) _{\infty }}\left( 1-q\right) ^{1-x}\text{, }x\neq
0,-1,-2,\cdots .  \label{Equation 5}
\end{equation}

which have the following properties:%
\begin{equation*}
\Gamma _{q}\left( x+1\right) =\left[ x\right] _{q}\Gamma _{q}\left( x\right) 
\text{, }\Gamma _{q}\left( 1\right) =1\text{ and }\lim_{q\rightarrow
1^{-}}\Gamma _{q}\left( x\right) =\Gamma \left( x\right) \text{, }\Re \left(
x\right) >0\text{.}
\end{equation*}

It has the $q$-integral representation as follows:%
\begin{equation}
\Gamma _{q}\left( s\right) =\int_{0}^{\frac{1}{1-q}}t^{s-1}E_{q}\left(
-qt\right) d_{q}t\text{.}  \label{Equation 6}
\end{equation}

When $\frac{\log \left( 1-q\right) }{\log q}\in 
\mathbb{Z}
$, becomes%
\begin{equation}
\Gamma _{q}\left( s\right) =\int_{0}^{\infty }t^{s-1}E_{q}\left( -qt\right)
d_{q}t\text{.}  \label{Equation 7}
\end{equation}

The $q$-Mellin transformation of a suitable function $f$ on $%
\mathbb{R}
_{q,+}$ is defined by%
\begin{equation}
M_{q}\left( f\right) \left( s\right) =\int_{0}^{\infty }t^{s-1}f\left(
t\right) d_{q}t  \label{Equation 8}
\end{equation}

(for details of Eqs. (14-17), see \cite{Kim 5}, \cite{Brahim}, \cite{Kac}).

In \cite{Rubin}, the novel $q$-differential operator was defined by Rubin as
follows:%
\begin{equation}
\partial _{q}\left( f\right) \left( x\right) =\frac{f\left( q^{-1}x\right)
+f\left( -q^{-1}x\right) -f\left( qx\right) +f\left( -qx\right) -2f\left(
-x\right) }{2\left( 1-q\right) x}\text{.}  \label{Equation 10}
\end{equation}

By (\ref{Equation 10}), we note that%
\begin{equation}
\lim_{q\rightarrow 1}\partial _{q}\left( f\right) \left( x\right) =f%
{\acute{}}%
\left( x\right) \text{.}  \label{Equation 17}
\end{equation}

By applying Rubin's $q$-differential operator to the generating function of $%
q$-Genocchi numbers and polynomials of higher order, we compute as follows:%
\begin{eqnarray*}
\sum_{n=0}^{\infty }\frac{z^{n}}{\left[ n\right] _{q}!}\partial
_{q}G_{n,q}^{\left( \alpha \right) }\left( x\right) &=&\partial _{q}\left\{
\left( \frac{\left[ 2\right] _{q}z}{e_{q}\left( z\right) +1}\right) ^{\alpha
}e_{q}\left( xz\right) \right\} \\
&=&\sum_{n=0}^{\infty }\left( \frac{1}{2\left( 1-q\right) }\sum_{l=0}^{n}%
\binom{n}{l}_{q}\left\{ 
\begin{array}{c}
q^{-l}+\left( -1\right) ^{l}q^{-l}-q^{l} \\ 
+\left( -1\right) ^{l}q^{l}+2\left( -1\right) ^{l}%
\end{array}%
\right\} x^{l-1}G_{n-l,q}^{\left( \alpha \right) }\right) \frac{z^{n}}{\left[
n\right] _{q}!}\text{.}
\end{eqnarray*}

By comparing the coefficients of $\frac{z^{n}}{\left[ n\right] _{q}!}$ on
both sides of the above equation. Then, we state the following theorem.

\begin{theorem}
Let $T_{q}\left( l\right) =$ $q^{-l}+\left( -1\right)
^{l}q^{-l}-q^{l}+\left( -1\right) ^{l}q^{l}+2\left( -1\right) ^{l}$, then we
get 
\begin{equation}
\partial _{q}G_{n,q}^{\left( \alpha \right) }\left( x\right) =\frac{1}{%
2\left( 1-q\right) }\sum_{l=0}^{n}\binom{n}{l}_{q}T_{q}\left( l\right)
x^{l-1}G_{n-l,q}^{\left( \alpha \right) }\text{.}  \label{Equation 11}
\end{equation}
\end{theorem}

\bigskip By (\ref{Equation 11}), we readily derive the following%
\begin{eqnarray*}
\partial _{q}G_{n,q}^{\left( \alpha \right) }\left( x\right) &=&\frac{1}{%
\left( 1-q\right) }\sum_{l=0}^{\left[ \frac{n}{2}\right] }\binom{n}{2l}%
_{q}\left\{ q^{-2l}+1\right\} G_{n-2l,q}^{\left( \alpha \right) } \\
&&+\frac{1}{\left( q-1\right) }\sum_{l=0}^{\left[ \frac{n-1}{2}\right] }%
\binom{n}{2l+1}_{q}\left\{ q^{2l+1}+1\right\} G_{n-1-2l,q}^{\left( \alpha
\right) }\text{.}
\end{eqnarray*}

Here $\left[ .\right] $ is Gauss' symbol. Consequently, we derive the
following theorem.

\begin{theorem}
\bigskip For any $n\in 
\mathbb{N}
^{\ast }$, we have%
\begin{eqnarray}
\partial _{q}G_{n,q}^{\left( \alpha \right) }\left( x\right) &=&\frac{1}{%
\left( 1-q\right) }\sum_{l=0}^{\left[ \frac{n}{2}\right] }\binom{n}{2l}%
_{q}\left\{ q^{-2l}+1\right\} G_{n-2l,q}^{\left( \alpha \right) }
\label{Equation 18} \\
&&+\frac{1}{\left( 1-q\right) }\sum_{l=0}^{\left[ \frac{n-1}{2}\right] }%
\binom{n}{2l+1}_{q}\left\{ q^{2l+1}+1\right\} G_{n-1-2l,q}^{\left( \alpha
\right) }\text{.}  \notag
\end{eqnarray}
\end{theorem}

\bigskip From (\ref{Equation 11}) and (\ref{Equation 18}), we conclude as
follows:

\begin{corollary}
The following identity%
\begin{eqnarray*}
\sum_{l=0}^{n}\binom{n}{l}_{q}T_{q}\left( l\right) x^{l-1}G_{n-l,q}^{\left(
\alpha \right) } &=&\sum_{l=0}^{\left[ \frac{n}{2}\right] }\binom{n}{2l}%
_{q}\left\{ 2q^{-2l}+2\right\} G_{n-2l,q}^{\left( \alpha \right) } \\
&&+\sum_{l=0}^{\left[ \frac{n-1}{2}\right] }\binom{n}{2l+1}_{q}\left\{
2q^{2l+1}+2\right\} G_{n-1-2l,q}^{\left( \alpha \right) }
\end{eqnarray*}%
is true.
\end{corollary}

By (\ref{Equation 11}), we have the following corollary.

\begin{corollary}
For any $n\in 
\mathbb{N}
^{\ast }$, we get%
\begin{equation*}
\lim_{q\rightarrow 1}\partial _{q}G_{n,q}^{\left( \alpha \right) }\left(
x\right) =nG_{n-1}^{\left( \alpha \right) }\left( x\right)
\end{equation*}%
where $G_{n}^{\left( \alpha \right) }\left( x\right) $ are called Genocchi
numbers and polynomials of higher order which are defined by the following
generating function:%
\begin{equation*}
\sum_{n=0}^{\infty }G_{n}^{\left( \alpha \right) }\left( x\right) \frac{t^{n}%
}{n!}=\left( \frac{2t}{e^{t}+1}\right) ^{\alpha }e^{xt}\text{ (see \cite%
{Jolany}, \cite{Rim}).}
\end{equation*}
\end{corollary}

We now give a $q$-analogue of D. Mili\^{c}ic's Lemma (see [22, page 1, Lemma
1.2.1]).

\begin{lemma}
Let $a_{n,q}$, $n\in 
\mathbb{N}
^{\ast }:=\left\{ 0,1,2,3,...\right\} $, be complex numbers such that $%
\sum_{n=0}^{\infty }\left\vert a_{n,q}\right\vert $ converges. Let%
\begin{equation*}
\lambda =\left\{ -n\mid n\in 
\mathbb{N}
^{\ast }\text{ and }a_{n,q}\neq 0\right\} \text{.}
\end{equation*}%
Then,%
\begin{equation*}
g_{q}\left( z\right) =\sum_{n=0}^{\infty }\frac{a_{n,q}}{\left[ z+n\right]
_{q}}
\end{equation*}%
converges absolutely for $z\in 
\mathbb{C}
-\lambda $ and uniformly on bounded subsets of $%
\mathbb{C}
-\lambda $. The $q$-function is a meromorphic function on complex plane with
simple poles at the points in $\lambda $ and Res$\left( g_{q},-n\right)
=a_{n,q}$ for any $-n\in \lambda $.
\end{lemma}

\begin{proof}
By using similar method in lecture notes of D. Mili\^{c}ic in [22], it is
clear that if $\left\vert \left[ z\right] _{q}\right\vert <R$, we see%
\begin{equation*}
\left\vert \left[ z+n\right] _{q}\right\vert =\left\vert \left[ z\right]
_{q}+q^{zn}\left[ n\right] _{q}\right\vert \geq \left\vert \left[ n\right]
_{q}-R\right\vert
\end{equation*}

for all $\frac{1-q^{n}}{1-q}>R$. Then, we get%
\begin{equation*}
\left\vert \frac{1}{\left[ z+n\right] _{q}}\right\vert \leq \frac{1}{\left[ n%
\right] _{q}-R}
\end{equation*}

for $\left\vert \left[ z\right] _{q}\right\vert <R$ and $\left[ n\right]
_{q}>R$. It follows that $\left[ n_{0}\right] _{q}>R$, we have%
\begin{equation*}
\left\vert \sum_{n=n_{0}}\frac{a_{n,q}}{\left[ z+n\right] _{q}}\right\vert
\leq \sum_{n=0}^{\infty }\frac{\left\vert a_{n,q}\right\vert }{\left\vert %
\left[ z+n\right] _{q}\right\vert }\leq \sum_{n=0}^{\infty }\frac{\left\vert
a_{n,q}\right\vert }{\left[ n\right] _{q}-R}\leq \frac{1}{\left[ n_{0}\right]
_{q}-R}\sum_{n=n_{0}}^{\infty }\left\vert a_{n,q}\right\vert
\end{equation*}

Hence, the series $\sum_{n>R}\frac{a_{n,q}}{\left[ z+n\right] _{q}}$
converges absolutely and uniformly on the disk $\left\{ z\mid \left\vert
z\right\vert <R\right\} $ and defines there a meromorphic function. It
follows that 
\begin{equation*}
\sum_{n=0}^{\infty }\frac{a_{n,q}}{\left[ z+n\right] _{q}}
\end{equation*}

is a meromorphic function on that disk with the simple poles at the points
of $\lambda $ in $\left\{ z\mid \left\vert z\right\vert <R\right\} $. Then,
for any $-n\in \lambda $, we have 
\begin{equation*}
g_{q}\left( z\right) =\frac{a_{n,q}}{\left[ z+n\right] _{q}}+\sum_{-j\in
\lambda -\left\{ n\right\} }^{\infty }\frac{a_{j,q}}{\left[ z+j\right] _{q}}=%
\frac{a_{n,q}}{\left[ z+n\right] _{q}}+\vartheta \left( z\right)
\end{equation*}

where $\vartheta \left( z\right) $ is holomorphic at $-n$. This also shows
that 
\begin{equation*}
Res\left( g_{q},-n\right) =a_{n,q}.
\end{equation*}

Thus, we successfully complete the proof of lemma. \ 
\end{proof}

\bigskip We now want to indicate that $\Gamma $ extends to a meromorphic
function by using Lemma 2. 9. That is, we discover the following%
\begin{equation*}
\Gamma _{q}\left( z\right) =\int_{0}^{\infty }t^{z-1}E_{q}\left( -qt\right)
d_{q}t=\int_{0}^{1}t^{z-1}E_{q}\left( -qt\right) d_{q}t+\int_{1}^{\infty
}t^{z-1}E_{q}\left( -qt\right) d_{q}t\text{.}
\end{equation*}

Then, the second integral converges for any complex $z$ and represents an
entire function. On the other hand, the $q$-exponential function is entire,
and we have%
\begin{eqnarray*}
\int_{0}^{1}t^{z-1}E_{q}\left( -qt\right) d_{q}t
&=&\int_{0}^{1}t^{z-1}\left\{ \sum_{j=0}^{\infty }\frac{\left( -1\right)
^{j}q^{\binom{j+1}{2}}}{\left[ j\right] _{q}!}t^{j}\right\} d_{q}t \\
&=&\sum_{j=0}^{\infty }\frac{\left( -1\right) ^{j}q^{\binom{j+1}{2}}}{\left[
j\right] _{q}!}\left\{ \int_{0}^{1}t^{z+j-1}d_{q}t\right\} \\
&=&\sum_{j=0}^{\infty }\frac{\left( -1\right) ^{j}q^{\binom{j+1}{2}}}{\left[
j\right] _{q}!}\frac{1}{\left[ z+j\right] _{q}}
\end{eqnarray*}

for any $z\in 
\mathbb{C}
$. Now also, we can write as follows:%
\begin{equation*}
\Gamma _{q}\left( z\right) =\int_{1}^{\infty }t^{z-1}E_{q}\left( -qt\right)
d_{q}t+\sum_{j=0}^{\infty }\frac{\left( -1\right) ^{j}q^{\binom{j+1}{2}}}{%
\left[ j\right] _{q}!}\frac{1}{\left[ z+j\right] _{q}}
\end{equation*}

for any $z$ in the right half plane. From Lemma 2. 9, the right hand-side of
the above identity defines a meromorphic function on the complex plane with
simple poles at $z=-j,$ $j\in 
\mathbb{N}
^{\ast }$. Then, we have the following theorem.

\begin{theorem}
For any $j\in 
\mathbb{N}
^{\ast }$, we derive the following%
\begin{equation}
\text{\textit{Res}}\left( \Gamma _{q},-j\right) =\frac{\left( -1\right)
^{j}q^{\binom{j+1}{2}}}{\left[ j\right] _{q}!}\text{.}  \label{Equation 19}
\end{equation}
\end{theorem}

As $q\rightarrow 1$ into (\ref{Equation 19}), we easily derive that%
\begin{equation*}
\lim_{q\rightarrow 1}\text{\textit{Res}}\left( \Gamma _{q},-j\right) =\frac{%
\left( -1\right) ^{j}}{j!}
\end{equation*}

which it is residue of Euler's Gamma function (see \cite{Milicic}).

Now also, by applying $q$-Mellin Transformation to generating function of $q$%
-Genocchi polynomials of higher order, then we compute as follows:%
\begin{eqnarray*}
\Im _{q}\left( z,x:\alpha \right) &=&\frac{1}{\Gamma _{\frac{1}{q}}\left(
z\right) }\int_{0}^{\infty }t^{z-\alpha -1}\left\{ \left( -1\right) ^{\alpha
}S_{q}\left( x,-t:\alpha \right) \right\} d_{\frac{1}{q}}t \\
&=&\sum_{l_{1},l_{2},...,l_{\alpha }=0}^{\infty }\left( -1\right)
^{l_{1}+l_{2}+...+l_{\alpha }}\left\{ \frac{1}{\Gamma _{\frac{1}{q}}\left(
z\right) }\int_{0}^{\infty }t^{z-1}E_{\frac{1}{q}}\left( -\frac{t}{q}\left(
qx+q\sum_{k=1}^{\alpha }l_{k}\right) \right) d_{\frac{1}{q}}t\right\} \\
&=&\left[ 2\right] _{q}^{\alpha }\sum_{l_{1},l_{2},...,l_{\alpha
}=0}^{\infty }\frac{q^{-z}\left( -1\right) ^{l_{1}+l_{2}+...+l_{\alpha }}}{%
\left( l_{1}+l_{2}+...+l_{\alpha }+x\right) ^{z}}\text{.}
\end{eqnarray*}

So, we now introduce definition of $q$-Hurwitz-Zeta type function as follows:

\begin{definition}
For any $z\in 
\mathbb{C}
$, then we define%
\begin{equation*}
\Im _{q}\left( z,x:\alpha \right) =\left[ 2\right] _{q}^{\alpha
}\sum_{l_{1},l_{2},...,l_{\alpha }=0}^{\infty }\frac{\left( -1\right)
^{l_{1}+l_{2}+...+l_{\alpha }}}{\left( qx+q\sum_{k=1}^{\alpha }l_{k}\right)
^{z}}\text{.}
\end{equation*}
\end{definition}

Via the above definition, we derive interpolation functions for $q$-Genocchi
numbers and polynomials of higher order at negative integers with the
following theorem.

\begin{theorem}
The following equality holds:%
\begin{equation*}
\Im _{q}\left( -n,x:\alpha \right) =\frac{q^{-n}G_{n+\alpha ,q}^{\left(
\alpha \right) }\left( x\right) }{\left[ \alpha \right] _{q}!\binom{n+\alpha 
}{\alpha }_{q}}\text{.}
\end{equation*}
\end{theorem}



\end{document}